\documentclass[12pt]{amsart}

\usepackage{amsmath}
\usepackage{amsfonts}
\usepackage{amssymb}
\usepackage{graphicx}
\usepackage{hyperref}
\usepackage{mathrsfs}
\usepackage{pb-diagram}
\usepackage{epstopdf}
\usepackage{amsmath,amsfonts,amsthm,enumerate,amscd,latexsym,curves}
\usepackage{bbm}
\usepackage[mathscr]{eucal}
\usepackage{epsfig,epsf,xypic,epic}
\usepackage{caption}
\usepackage{subcaption}

\newtheorem{theorem}{Theorem}[section]
\newtheorem{lemma}[theorem]{Lemma}

\newtheorem{prop}[theorem]{Proposition}
\newtheorem{defn}[theorem]{Definition}

\newtheorem{remark}[theorem]{Remark}
\newtheorem{conjecture}[theorem]{Conjecture}

\numberwithin{equation}{section}

\newcommand{\Z}{\mathbb{Z}}

\newcommand{\C}{\mathbb{C}}

\newcommand{\AI}{A_\infty}
\newcommand{\RP}{\mathbb{R}P}
\newcommand{\CP}{\mathbb{C}P}
\newcommand{\bL}{\mathbb{L}}

\newcommand{\Id}{{\mathrm{Id}}}
\newcommand{\Fuk}{{\sf Fuk}}
\newcommand{\MF}{{\sf MF}}
\newcommand{\K}{\mathbb{K}}
\newcommand{\m}{\mathfrak{m}}
\newcommand{\OC}{\mathcal{OC}}

\newcommand{\WT}[1]{\widetilde{#1}}

\newcommand{\one}{\mathbf{1}}

\title{Ungraded matrix factorizations as mirrors of non-orientable Lagrangians}

\author[Amorim]{Lino Amorim}
\address{Department of Mathematics\\ Kansas State University\\ 138 Cardwell Hall, 1228 N. 17th Street\\
	Manhattan, KS 66506\\ USA}
\email{lamorim@ksu.edu}
\author[Cho]{Cheol-Hyun Cho}
\address{Department of Mathematical Sciences, Research institute of Mathematics\\ Seoul National University\\ San 56-1, 
	Shinrimdong\\ Gwanakgu \\Seoul 47907\\ Korea}
\email{chocheol@snu.ac.kr}

\begin{document}

	\begin{abstract}
	We introduce the notion of ungraded matrix factorization as a mirror of non-orientable Lagrangian submanifolds.
	An ungraded matrix factorization of a polynomial $W$, with coefficients in a field of characteristic 2, is a  square matrix $Q$ of polynomial entries 
	satisfying $Q^2 = W \cdot \mathrm{Id}$. We then show that non-orientable Lagrangians correspond to ungraded matrix factorizations under the localized mirror functor and illustrate this construction in a few instances. Our main example is the Lagrangian submanifold $\mathbb{R}P^2 \subset \mathbb{C}P^2$ and its mirror ungraded matrix factorization, which we construct and study. In particular, we prove a version of Homological Mirror Symmetry in  this setting.  	
	\end{abstract}

	\maketitle

\section{Introduction}
An extended version of the Homological Mirror Symmetry conjecture relates the Fukaya category of a symplectic manifold
with the matrix factorization category of a mirror Landau-Ginzburg model.
Here, we restrict ourselves to the cases that the Landau-Ginzburg model is given by a (Laurent) polynomial $W$.
Matrix factorization category is a classical subject, extensively studied in 80's under the name of maximal Cohen-Macaulay modules
of Cohen-Macaulay rings, and later shown to be  derived equivalent to the category of singularities.

In this paper, we would like to bring attention to  an unexplored aspect of this mirror correspondence.
In the study of homological mirror symmetry  the objects of Fukaya category are usually assumed to be  orientable Lagrangians.
But it has been already known that monotone non-orientable Lagrangians (such as $\RP^n \subset \CP^n$ for even $n$)
have non-trivial Floer cohomology with coefficients in a field of characteristic two. But its mirror counterpart has not been studied.

First, we will introduced the notion of an ungraded matrix factorization as a mirror of non-orientable Lagrangian submanifold.
Recall that a ($\Z/2$-graded) matrix factorization of $W$ is given by 
a pair $(Q_0,Q_1)$ of polynomial matrices, say of size $k \times k$,
satisfying  
\begin{equation}\label{eq:qq}
Q_0 \cdot Q_1 = Q_1 \cdot Q_0 = W \cdot \textrm{Id}_k.
\end{equation}
We may write it as the following block matrix
\begin{equation}\label{eq:q0q1}
Q = \left(\begin{array}{cc}
 0    &Q_0   \\
 Q_1   &  0
\end{array}\right),  Q^2 = W \cdot \textrm{Id}_{2k}.
\end{equation}

We consider a polynomial $W$ with coefficients in a field of characteristic 2 and define an {\textbf{ungraded matrix factorization}} as a polynomial matrix $Q$, say of size $n$ with  $Q^2 = W \cdot \textrm{Id}_n$.
Here $n$ does not have to be even, and hence it may not be of the block form of \eqref{eq:q0q1}. 

To understand why this should be a mirror of a non-orientable Lagrangian submanifold, 
recall that  morphisms in the Fukaya category are generated by (transverse) intersections of two Lagrangians,
and if both Lagrangians are oriented, such an intersection carries a canonical $\Z/2$-grading:
the direct sum of tangent spaces of two oriented Lagrangians are compared to 
the orientation of the ambient symplectic manifold. Such a $\Z/2$-grading is lost if at least one of the Lagrangian is non-orientable.

From an ungraded matrix factorization $Q'$ of size $n$,
we can define a doubling $\mathfrak{D}(Q')$, which is a $\Z/2$-graded matrix factorization 
 \begin{equation} 
\mathfrak{D}(Q')= \left(\begin{array}{cc}
 0    &Q'   \\
 Q'   &  0
\end{array}\right).
\end{equation}
The homological mirror of the doubling of an ungraded matrix factorization is in fact an orientation double cover!
\begin{prop}
If an non-orientable monotone Lagrangian $L$ is homological mirror to the ungraded matrix factorization $Q'$ under the localized mirror functor \cite{CHL17},
then the orientation double cover of $L$ is homological mirror to the doubling $\mathfrak{D}(Q')$.
\end{prop}
Only certain functions $W$ admit such a matrix factorization  $(Q_0,Q_1)$ with  $Q_0 = Q_1$, and under
the mirror correspondence, this may be related to the existence of an non-orientable Lagrangian.

We will illustrate this in the example of $A_n$-singularity, whose explicit homological mirror symmetry
has been discussed by the second author together with Choa and Jung \cite{CCJ}

\begin{theorem}
 There is a non-orientable non-compact Lagrangian $K$ in the Milnor fiber quotient of $A_n$-singularity (for $n=2k-1$) $W=x^{2k} + y^2$
 which is mirror 
to the following ungraded matrix factorization  
 $Q= \left(\begin{array}{cc}
  y  &  x^{k}\\
 x^{k} & -y    
   \end{array}\right)$
   of the mirror transpose polynomial $W^T = W$.
\end{theorem}
\begin{remark}
For ADE singularity, there are finitely many indecomposable matrix factorizations (see \cite{Yo}).
$\mathfrak{D}(Q)$ for $A_{2k-1}$ singularity is the only indecomposable matrix factorization that can be written as a doubling.
\end{remark}
%
%
%
%
%

We will now describe our results on our main example, the real projective space $\mathbb{R}P^2$ in $\mathbb{C}P^2$.
Floer theory of non-orientable Lagrangians is well-defined if we use coefficients in a field  $\mathbb{K}$ with a characteristic two, such as an extension field of $\Z/2$.
We find the homological mirror of real projective spaces using the localized mirror functor defined in \cite{CHL-toric} and the computations from \cite{AA}.
\begin{theorem}
A $\mathbb{R}P^n \subset \CP^n$ for $n$ even is mirror to an ungraded matrix factorization under
the localized mirror functor of \cite{CHL-toric}. 

In particular, the mirror to $\mathbb{R}P^2$ is the following ungraded matrix factorization of $W = x + y + \frac{1}{xy}$:
\begin{equation}\label{eq:1}
Q_{\mathbb{R}P^2}= \left(\begin{array}{cccc}
 0    &1 &  1 & x^{-1}y^{-1}    \\
 y    &   0  & x^{-1} & 1    \\
 x    &   y^{-1}  & 0 & 1    \\
  1    &   x  & y & 0    \end{array}\right) 
\end{equation}
\end{theorem}
 
It was shown by Tonkonog \cite{ton18} (for all $n$) and Evans-Lekili \cite{EL} (for odd $n$ by a different method) that $\mathbb{R}P^n$ generates Fukaya category of $\mathbb{C}P^n$ of
potential value $0$.   The proof of Tonkonog was based on the following computation of
the closed-open string map being injective (hence an isomorphism).
$$QH^*(\mathbb{C}P^2) \to HF(\mathbb{R}P^2,\mathbb{R}P^2)$$

In this paper, we verify a mirror of this statement for the above ungraded matrix factorization.
The mirror of $\CP^2$ is the well known Laurent polynomial $W=x+y+x^{-1}y^{-1}$ (called Givental-Hori-Vafa potential. See \cite{CO} for a Floer theoretic construction of this mirror). 
The $B$-model analogue of the quantum cohomology ring $QH^*(\mathbb{C}P^2)$  is given by the Jacobian ring.
$$Jac(W) = \frac{ \mathbb{K}(x,y)}{(\partial_xW, \partial_y W)}$$

\begin{theorem}\label{thm:main}
Let $Q_{\mathbb{R}P^2}$ be the ungraded matrix factorization \eqref{eq:1} which is mirror to $\RP^2$.
Then the closed-open map $$\mathcal{C}:Jac(W) \to Hom_{H(\mathcal{MF}^{un}(W))}(Q_{\mathbb{R}P^2}, Q_{\mathbb{R}P^2}) $$ sending $\alpha \to \alpha \cdot \Id$
is an isomorphism.
\end{theorem}

Our main result is the following version of Homological Mirror Symmetry in this example.

\begin{theorem}\label{thm:main2}
The localized mirror functor induces an embedding of the derived categories
\[D^\pi\mathcal{F}^\bL: D^\pi\Fuk_0(\CP^2) \to D^\pi\MF^{un}(W=x+y+\frac{1}{xy}).\]
\end{theorem}
We except this functor to be an equivalence - see Section 4 for more details.	

We end this introduction by mentioning some standard features of Fukaya categories and categories of matrix factorizations that fail in our setting - ungraded and over a field of characteristic two. We believe these are very interesting phenomena that deserve further investigation.

We explained that Fukaya category of monotone Lagrangian submanifolds should be considered for each potential value $\lambda$ separately, and it was denoted by $ \Fuk_\lambda(M)$.
It is well-known (see \cite{She_fano} for a proof) that in characteristic 0, $ \Fuk_\lambda(M)$ is non-trivial only when $\lambda$ is an eigenvalue of the $c_1(TM)$ action on quantum cohomology of $M$. 
Similarly, it is well-known that over characteristic 0, the category of matrix factorization of $W$ should be considered separately for each critical value $\lambda$ as a matrix factorization of $W -\lambda$.
Both of these breaks down over characteristic two, and let us explain this phenomenon.

First, let us consider the case of  matrix factorizations, when $\lambda=0$ for simplicity.
When working over a field of $char\neq 2$, the identity 
$$d_{Q,Q}(Q)=2W \cdot \textrm{Id}_E \;\;\;\;  \Longrightarrow \;\;\;\; d_{Q,Q}(\frac{1}{2W}Q)=\textrm{Id}_E$$ after inverting $W$. This implies that the support of $(E, Q)$ is contained in $W^{-1}(0)$. Moreover, Orlov \cite{Orl} showed that the cohomology of the matrix factorization category is equivalent to $D^b_{sing}(W^{-1}(0))$, the singularity category of $W^{-1}(0)$.

Theorem \ref{thm:main} gives an example of an ungraded matrix factorization (in characteristic 2) whose support is contained in the critical locus of $W$, but it is not contained in $W^{-1}(0)$. In fact $W^{-1}(0)$ is smooth, therefore Orlov's result cannot apply to this case.
It would be interesting to know if the cohomology of the ungraded matrix factorization category admits an alternative description as some kind of singularity category.

Second, the argument in the case of Fukaya category rests on the identity that over characteristic 0, the Maslov cycle can be represented as  the twice of the first Chern class $c_1(TM)$, see \cite{She_fano} for details.
But fails in characteristic two for non-orientable Lagrangians. It would be very interesting to find a new criterion for which values $\lambda$ give non-trivial components of the Fukaya category in this setting.

This paper is organized as follows: in Section 2 we introduce the category of ungraded matrix factorizations; in Section 3 we discuss the localized mirror functor in characteristic two; in Section 4 we consider our main example $\mathbb{R}P^2 \subset\mathbb{C}P^2$ and prove Theorems 1.4, 1.5 and 1.6; finally in Section 5 we consider the $A_n$-singularity example.

\section{Category of ungraded matrix factorizations}
Let $\mathbb{K}$ be a  field of characteristic two, and let
$A =\mathbb{K}[z_1^{\pm 1}, \cdots,z_n^{\pm 1}],$
be the ring of Laurent polynomials (in $n$ variables) over $\mathbb{K}$.
Recall \cite{Orlov} that a $\Z/2$-graded matrix factorization of $W\in A$ is  given by a $\Z/2$-graded finite dimensional free $A$-module $E=E^0\oplus E^1$
together with an odd map $Q: E^\bullet \to E^{\bullet+1}$ such that $Q^2 = W \cdot \textrm{Id}_E$.

We now define the differential category of ungraded matrix factorizations. By differential category we mean a category for which the set of morphisms $Hom(X,Y)$ are $\mathbb{K}$-vector spaces equipped with a linear map $d_{X,Y}:Hom(X,Y)\to Hom(X,Y)$ which satisfies $d_{X,Y}^2=0$. Moreover composition satisfies 
$$d_{X,Z}(g\circ f)= d_{Y,Z}(g)\circ f + g\circ d_{X,Y}(f),$$
and the identity morphisms are closed: $d_{X,X}(\one_X)=0$. In other words, it is the same as a dg-category if one drops the grading conditions.

To a differential category $\mathcal{M}$ we can associate its homotopy category $H(\mathcal{M})$. This is a $\mathbb{K}$-linear category with the same objects as $\mathcal{M}$ and morphisms given by 
$$Hom_{H(\mathcal{M})}(X, Y):= H(Hom(X,Y))= \frac{Ker(d_{X,Y})}{Im(d_{X,Y})}.$$

\begin{defn}
 An \emph{ungraded matrix factorization} of $W \in A$ is a pair $(E,Q)$  consisting of a  finite dimensional free $A$-module $E$ and an $A$-module homomorphism $Q: E \to E$ satisfying $$Q^2 = W \cdot \textrm{Id}_E.$$ 
A morphism  between two ungraded matrix factorizations $(E, Q), (F, R)$ is simply an $A$-module homomorphism $f:E \to F$. On the space of morphisms $Hom((E, Q), (F, R))$ we define the differential $d_{Q,R}$ by the formula
$$d_{Q,R} (f) = R \circ f +  f \circ Q.$$

We denote by $\mathcal{MF}^{un}(W)$ the differential category whose objects are ungraded matrix factorizations of $W$ and morphisms, as defined above, are composed in the obvious way.
\end{defn}

It is a simple exercise to check that $\mathcal{MF}^{un}(W)$ is in fact a differential category. Note that we do not need the usual sign factor $(-1)^{|f|}$ in the definition of $d$ since we are in characteristic two.
 
Let us compare $\mathcal{MF}^{un}(W)$ with the standard category of graded matrix factorizations. In both categories, any matrix factorization is supported on $Crit(W)$ the critical locus of $W$. This means that the module $H(Hom((E, Q), (E, Q)))$ is supported on $Crit(W)$. To see this we differentiate the identity $Q\circ Q =W\cdot \textrm{Id}_E$ with respect  to the variable $z_i$ to obtain 
\begin{equation}\label{eq:Qdiff}
\frac{\partial Q}{\partial z_i} \circ Q + Q \circ \frac{\partial Q}{\partial z_i} =  \frac{\partial W}{\partial z_i}\cdot \textrm{Id}_E.
\end{equation}
Away from $Crit(W)$, we can invert $ \frac{\partial W}{\partial z_i}$, and the above equation gives $d_{Q,Q}( (\frac{\partial W}{\partial z_i})^{-1}  \frac{\partial Q}{\partial z_i})= \textrm{Id}_E$, which implies that the $d_{Q,Q}$ cohomology vanishes.

On the other hand, since we are working over a field of $char=2$, the support of $(E, Q)$ is not necessarily contained in $W^{-1}(0)$.  In the case, when $char\neq 2$, the identity $d_{Q,Q}(Q)=2(W\textrm{Id}_E$ gives $d_{Q,Q}(\frac{1}{2W}Q)=\textrm{Id}_E$ after inverting $W$. This implies that the support of $(E, Q)$ is contained in $W^{-1}(0)$. This argument obviously fails in characteristic two and we will see an example where the support is entirely outside $W^{-1}(0)$.

Another similarity is that, just like for standard matrix factorizations, the Jacobian ring (or Milnor ring) of $W$, $\displaystyle\textrm{Jac}(W):=\frac{A}{\langle\frac{\partial W}{\partial z_1},\ldots \frac{\partial W}{\partial z_n} \rangle}$ acts on the homotopy category of ungraded matrix factorizations.

\begin{lemma}
For any pair of objects $(E, Q)$, $(F,R)$ in $\mathcal{MF}^{un}(W)$, the vector space 
$$H(Hom((E, Q),(F, R)))$$ is a module over $\textrm{Jac}(W)$.
\end{lemma}
\begin{proof}
The standard proof still applies. Given a closed morphism $f\in Hom((E, Q),(F, R))$, compose it on the left of both sides of Equation (\ref{eq:Qdiff}). Using the condition $R\circ f=f\circ Q$ we obtain $d_{R,Q}(f\circ \frac{\partial Q}{\partial z_i})= \frac{\partial W}{\partial z_i}f$.
This shows that the action is well-defined on the homotopy category.
\end{proof}

Next we observe that there are functors between $\Z/2$-graded and ungraded matrix factorizations. There is an obvious functor $\mathfrak{F}: \mathcal{MF}(W) \to \mathcal{MF}^{un}(W)$ which simply forgets the grading. It induces a functor on the cohomology level $\mathfrak{F}: H^0(\mathcal{MF}(W)) \to H(\mathcal{MF}^{un}(W))$, which we still denote by $\mathfrak{F}$.  On the other direction there is a \emph{doubling} functor $\mathfrak{D}: H(\mathcal{MF}^{un}(W)) \to H^0(\mathcal{MF}(W))$, which sends $(E,Q)$ to $(E^\bullet:= E\oplus E, Q_\bullet)$ with $Q_\bullet= \left(\begin{array}{cc}
0    & Q \\
Q & 0    
\end{array}\right)$.  On the level of morphisms $\mathfrak{D}(f):=\left(\begin{array}{cc}
f    & 0 \\
0 & f    
\end{array}\right)$.

The following is easy to check.
\begin{lemma}
$\mathfrak{D}$ is both left and right adjoint to $\mathfrak{F}$. That is, there are natural isomorphisms:
$$Hom_{H(\mathcal{MF}^{un}(W))}(\mathfrak{F}(X), Y) \cong Hom_{H^0(\mathcal{MF}(W))}(X, \mathfrak{D}(Y))$$
$$Hom_{H(\mathcal{MF}^{un}(W))}(Y,\mathfrak{F}(X)) \cong Hom_{H^0(\mathcal{MF}(W))}(\mathfrak{D}(Y), X),$$
for any matrix factorizations $X$ in $\mathcal{MF}(W)$ and $Y$ in $\mathcal{MF}^{un}(W)$.
\end{lemma}

\section{Fukaya category and localized mirror functor in characteristic two}

When working in characteristic zero, in order to define the Fukaya category $\Fuk(M)$ of a symplectic manifold (or orbifold) $M$, 
one needs to orient coherently the relevant moduli spaces of pseudo-holomorphic polygons. To guarantee the existence of these orientations one is forced to impose orientability conditions on the Lagrangians themselves \cite{FOOO}.

In our setup, working over a field of characteristic two, there are no signs to be assigned and one can ignore the problem of orientations. Then one can include non-orientable Lagrangians in the Fukaya category. However, working in positive characteristic requires us to avoid using virtual techniques that lead to rational counts when dealing with the moduli spaces of pseudo-holomorphic maps. Fukaya--Oh--Ohta--Ono \cite{FOOOZ} have shown that this can be done if one restricts to \emph{spherically positive} symplectic manifolds. These are symplectic manifolds with compatible almost complex structures for which all pseudo-holomorphic spheres with non-positive Chern number are constant. Examples of this include: monotone symplectic manifolds and exact symplectic manifolds (or orbifold quotients of these). 

In this paper we will restrict to the monotone and exact cases. This has the added advantage of allowing us to avoid Novikov rings and simply work over a field $\K$ of characteristic two. A convenient setup for the monotone Fukaya category is given in \cite{She_fano} for example.

Objects in the Fukaya category are given by pairs $(L,b)$, where $L$ is a Lagrangian (monotone or exact) and $b$ is a weak bounding cochain (see \cite{FOOO, She_fano} for details). For such pair the corresponding curvature term $\m_0(L,b)=\mathfrak{P}(b) \one_L$, where $\one_L$ is the identity element and $\mathfrak{P}(b) $ is an element in $\K$, called the disk potential. In order to obtain a genuine $A_\infty$ category, as opposed to a curved one, one needs to fix the value of the potential. This means we have the decomposition
\begin{equation}
	\Fuk(M) = \Pi_\lambda \Fuk_\lambda(M),
\end{equation}
where objects in $\Fuk_\lambda(M)$ are pairs $(L,b)$ with $\mathfrak{P}(b)=\lambda$. We would like to point out that, by design, this $\AI$-categories are ungraded, meaning the Hom spaces in these categories are just vector spaces (not graded). Since we are working over a field of characteristic 2, this causes no problems, as all the signs that appear in the $\AI$ equations, and which usually depend on the degrees of the various morphisms, can be ignored. Therefore all the usual notions of $\AI$-category carry over to the ungraded (and characteristic 2) setting without changes.



The second author together with Hong and Lau developed a localized mirror functor formalism \cite{CHL17}, \cite{CHL-toric}, which gives a geometric
homological mirror symmetry $\AI$-functor using Lagrangian deformation theory. Let us very briefly review this construction. 
We first fix a reference Lagrangian $\bL$, for which there is a family of bounding cochains parametrized by a $\K$-vector space $E$, that is for each $x \in E$, there is a bounding cochain $b(x)$ for the Lagrangian $\bL$. The disk potential then defines a function on $E$, which we denote by $W: E \to \K$, by setting $W(x)=\mathfrak{P}(b(x))$. We assume that $W$ is a Laurent polynomial. 
The localized mirror functor is an $A_\infty$-functor 
$$\mathcal{F}^\bL: \Fuk_\lambda(M) \to \MF(W-\lambda).$$
On the level of objects, this functor sends an object $(L,b)$, where $L$ is assumed to be transverse to $\bL$, to the deformed Floer complex $$\mathcal{F}^\bL(L,b):=CF\big((\bL,b(x)),(L,b); \m_1^{b(x),b}\big).$$
The $A_\infty$ equations imply $(\m_1^{b(x),b})^2=(W(x)-\lambda)Id$, which shows $\mathcal{F}^\bL(L,b)$ is a matrix factorization. The only difference, from the standard case is that, since $L$ is non-orientable, the Floer complex is not (in general) $\Z/2$-graded.

From the previous discussion and our definition of ungraded matrix factorization, we obtain the following result.
\begin{prop}
The localized mirror functor $\mathcal{F}^\bL: \Fuk_\lambda(M) \to \MF^{un}(W-\lambda)$ is an $A_\infty$-functor which takes non-orientable Lagrangians to
ungraded matrix factorizations.
\end{prop}

 We conclude this section by giving a construction on the Fukaya category which is an analogue of the doubling functor $\mathfrak{D}$ introduced in the previous section.
 Consider a non-orientable monotone Lagrangian $L$ in $M$, equipped with the trivial bounding cochain. Let  $\WT{L}$ be the orientation double cover of $L$. The covering map $\pi : \WT{L} \to L \subset M$ defines a Lagrangian immersion,
 and $\WT{L}$ may be regarded as an immersed Lagrangian in $M$. One can consider $\WT{L}$ as a Lagrangian $L$ equipped with a $\Z/2$-local system (the push-forward of the constant $\Z/2$-local system in $\WT{L}$).

\begin{prop}\label{prop:d}
For a monotone non-orientable Lagrangian $L$, let us denote by $\WT{L}$ its orientation double cover. 
If  $L$ is homological mirror to an ungraded matrix factorization $Q$ under the localized mirror functor $\mathcal{F}^\bL$, then $\WT{L}$ is homological mirror 
to its doubling $\mathfrak{D}(Q)$.
\end{prop}
\begin{proof}
A generator of $CF(\bL,\WT{L})$ for the double cover is given by a generator $p \in CF(\bL, L)$ together with the choice of its lifting $\WT{p}$  for the orientation
double covering $\WT{L} \to L$.
Denote by $F$ the $A$-module $CF(\bL, L)\otimes A$ for $Q$, 
we have $E =E_0 \oplus E_1:= F \oplus F$ is the $A$-module for $\mathfrak{D}(Q)$.
Now, observe that for each pair $(p,\WT{p})$, we have a canonical $\Z/2$-grading by comparing  the orientation of $T_{\WT{p}}\WT{L} \oplus T_p \bL$ and $T_pM$ (even if they agree, odd otherwise).
It is left to argue that the differential $\m_1^{b(x),b}$ switches $\Z/2$-grading.
Recall that the Floer differential counts Maslov-Viterbo index one strips, and the differential has odd degree. This remains true even when we add insertions of bounding cochains of $\bL$.  Thus, the counting of rigid strips (which contributed a map of $F \to F$) now defines the same counting but  contributing to the following two  $A$-module maps $E_0 \to E_1, E_1 \to E_0$. This proves the proposition.
\end{proof}

\begin{remark}
	This construction should be promoted to a functor
	\[\mathfrak{D}:\Fuk_\lambda(M)\to \Fuk_\lambda^{or}(M),\]
where  $\Fuk_\lambda^{or}(M)$ is the Fukaya category consisting of orientable Lagrangians or non-orientable Lagrangians equipped with orientations local systems as above. We then expect the localized mirror functor to intertwine the two doubling functors. 
\end{remark}

\section{First example: the projective plane}

As our main example, we will describe the mirror of real projective spaces. We consider $\CP^n$, equipped with a monotone symplectic form, and for our reference Lagrangian $\bL$ we take the Clifford torus. We consider a family of bounding cochains on $\bL$, which can be identified with a $\K$ local system on the torus. Such local system, is of course determined by its holonomy and therefore by a tuple $(x_1,\ldots, x_n) \in (\K^*)^n$. As computed in \cite{CO}, the disk potential is given by $W(x_1, \ldots, x_n)= x_1 +\ldots+ x_n+ \frac{1}{x_1\cdots x_n}$.

We will now apply the localized mirror functor formalism to this case. In particular, we are interested in the Lagrangian submanifold $\mathbb{R}P^n \subset \CP^n$. For this Lagrangian one can take the trivial bounding cochain $b=0$, which has potential value $\mathfrak{P}(0)=0$. Moreover, for this bounding cochain, $HF(\mathbb{R}P^n, b=0)$  the self Floer cohomology (in other words, the cohomology of the endomorphism algebra of this object in the Fukaya category) is isomorphic to the singular cohomology $H^*(\mathbb{R}P^n, \K)$ as vector spaces, see \cite{Oh2} for example. Therefore $(\mathbb{R}P^n, b=0)$ is a non-trivial object in $\Fuk_0(\CP^n)$. We want to describe the image of this Lagrangian under the mirror functor
$$\mathcal{F}^\bL: \Fuk_0(\CP^n) \to \MF^{un}(x_1 +\ldots+ x_n+ \frac{1}{x_1\cdots x_n}).$$
In the case of $n$ odd, when $\mathbb{R}P^{n}$ is orientable, this was done in \cite{CHL-toric}, based on the Floer complex computations of the first author and Alston \cite{AA}. Here, using the same computations, we consider the even case, which leads to ungraded matrix factorizations. For simplicity we  state the $n=2$ case.

\begin{theorem}\label{thm:mirror_rp2}
	The mirror to $\mathbb{R}P^2$, that is $\mathcal{F}^\bL(\mathbb{R}P^2,b=0)$ is the following ungraded matrix factorization of $W = x + y + \frac{1}{xy}$:
	\begin{equation}\label{eq:1}
		Q_{\mathbb{R}P^2} = \left(\begin{array}{cccc}
			0    &1 &  1 & x^{-1}y^{-1}    \\
			y    &   0  & x^{-1} & 1    \\
			x    &   y^{-1}  & 0 & 1    \\
			1    &   x  & y & 0    \end{array}\right) 
	\end{equation}
\end{theorem}
\begin{proof}
	\begin{figure}[h]
		\includegraphics[scale=0.4]{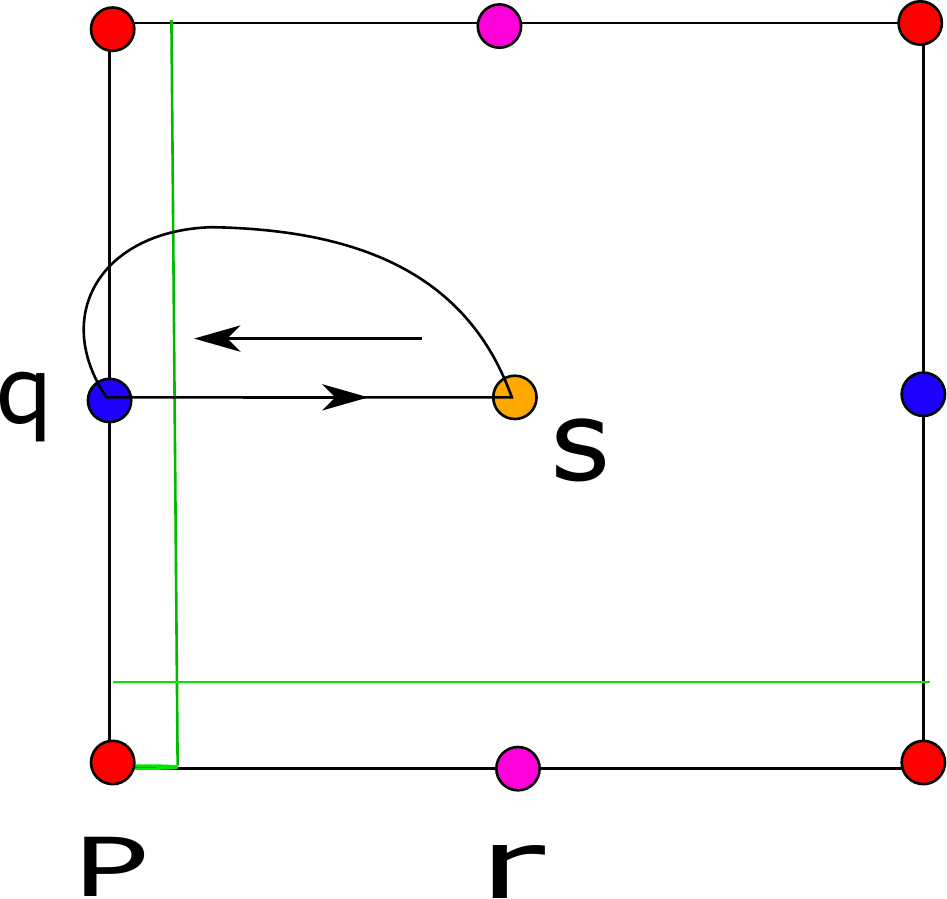}
		\centering
		\caption{Localized mirror functor using hyper-tori }
		\label{T2}
	\end{figure}
	
	Let $\bL$ be the Clifford torus in $\mathbb{C}P^2$, given by 
	$$\bL = \{[z_0,z_1,z_2] \mid |z_0|=|z_1|=|z_2|\} = \{[1,e^{i\theta_1},e^{i\theta_2}] \mid 0 \leq \theta_i < 2\pi \}$$
	Then, the intersection $\mathbb{R}P^2 \cap \mathbb{C}P^2$ is given by 4 points:
	$$ p =[1,1,1], q = [1,1,-1], r = [1,-1,1], s = [1,-1,-1]$$
	The rectangle in Figure \ref{T2} with opposite edges identified, illustrates the torus $\bL$ with these intersection points.
	The matrix factorization is given by $\m_1^{b(x),0}$ in $CF^*(\bL, L)$. We will write it as a matrix, using the basis $(s,r,q,p)$.
	Here the effect of the bounding cochain $b(x)$ is given by the intersection with hyper-tori, which represent the corresponding local system. In this case, two
	hyper-tori  $H_1=[1,e^{i\epsilon},U(1)]$,  $H_2=[1,U(1),e^{i\epsilon}]$, are also illustrated in the Figure as green lines.
	Counting a decorated strip means that we count rigid $J$-holomorphic strips for the Floer complex as usual, but
	in addition, we look at the intersection of the  boundary path on $\bL$ of such a strip, and take the signed intersection number
	with hyper-tori $H_1$ and $H_2$, say $i_1, i_2$. Then we give a weight to the counting of the strip by $x^{i_1}y^{i_2}$.

	The rigid holomorphic strips between $\bL$ and $L$ where described in \cite{AA}: since $\bL$ is invariant under the standard anti-holomorphic involution(which fixes $\RP^2$), any such strip can be	doubled to a holomorphic disc with boundary on $\bL$ of Maslov index two. Such holomorphic discs are already classified by \cite{C}, which are nothing but
	the standard ones $[z,1,1], [1,z,1], [1,1,z]$. For example, halves of the disc $[1,z,1]$ defines strips between $p=[1,1,1]$ and $r=[1,-1,1]$ and  also between $q=[1,1,-1]$ and $s=[1,-1,-1]$.
	
	In Figure \ref{T2}, we illustrated this strip from $s$ to $q$, where  the boundary given by the straight line lies in $\bL$, and the curved boundary lies in $\mathbb{R}P^2$. Note that boundary of this strip intersects the hyper-tori $H_1$ (with the orientation as in the Figure). Thus it should be counted as $x$.
	As we are using $(s, r,q,p)$ as an ordered basis, this counting gives the entry $x$ in $(3,1)$-entry of $Q$.
	
	On the other hand, the same disc $[1,z,1]$ contributes a strip from $s$ to $q$ as well, but as its boundary on $\bL$ does not
	intersect $H_2$ nor $H_1$. This gives $1$ in the $(1,3)$-entry of $Q$.
	Applying the same procedure to all intersections and discs, we obtain the entire matrix $Q$.
\end{proof}

\begin{remark}
	The mirror of $\mathbb{R}P^n \subset \CP^n$ for $n$ even can be computed analogously and we leave it as an exercise. In fact, the description of the Floer complex in \cite{AA} can be used to compute the mirror for the real Lagrangian in any monotone toric manifold.
\end{remark}

Now we compute the endomorphism algebra of the matrix factorization $Q_{\mathbb{R}P^2}$.

\begin{theorem}\label{thm:mf_rp2}
	The map $\mathcal{C}:Jac(W) \to Hom_{H(\mathcal{MF}^{un}(W))}(Q_{\mathbb{R}P^2}, Q_{\mathbb{R}P^2}) $ sending $\alpha \to \alpha \cdot \Id$
	is an isomorphism. In particular, we have
	\[Hom_{H(\mathcal{MF}^{un}(W))}(Q_{\mathbb{R}P^2}, Q_{\mathbb{R}P^2}) \cong  \frac{\K[x]}{x^3-1}.\] 
\end{theorem}

The proof of this theorem is a rather long computation that we postpone to Section 6.

It was proved by Haug in \cite{Hau} that $HF(\mathbb{R}P^2,b=0)$ is isomorphic (as an ungraded algebra) to the quantum cohomology of $\mathbb{C}P^2$, in characteristic two. Explicitly, $HF(\mathbb{R}P^2,b=0)$ is isomorphic to  $\displaystyle \frac{\K[\alpha]}{\alpha^3-1}$, where $\alpha$, under the vector space isomorphism with the singular cohomology $H^*(\mathbb{R}P^2, \K)$, corresponds to a generator of $H^1(\mathbb{R}P^2, \K)$. Therefore, by the above theorem, this is isomorphic to the endomorphism algebra of its mirror. We will now see that the localized mirror functor induces this isomorphism. 

\begin{prop}\label{prop:FLiso}
	The localized mirror functor induces an isomorphism $$\mathcal{F}^\bL :HF(\mathbb{R}P^2,b=0) \to {\mathrm End}_{H(\mathcal{MF}^{un}(W))}(Q_{\mathbb{R}P^2}).$$
	Explicitly, under the isomorphism in Theorem \ref{thm:mirror_rp2}, we have $\mathcal{F}^\bL(\alpha)=x^2$. 
\end{prop}
\begin{proof}
	By definition we have $\mathcal{F}^\bL(\alpha)(p)=\m_2^{b(x),0,0}(p, \alpha)$, for $p \in CF(\bL, \mathbb{R}P^2)$. In order to perform this calculation it is more convenient to use a \emph{Morse-Bott} model for the Floer complexes, either by using a Morse function on $\mathbb{R}P^2$ and counting pearly trajectories \cite{Hau} or using smooth singular chains on $\mathbb{R}P^2$, as in \cite{FOOOZ}. Here we do the later. We represent (the Poincare dual) of $\alpha$ by the closed curve $\gamma(t)=[0:\cos{t}:\sin{t}]$, $t\in [-\pi/2, \pi/2]$ in $\mathbb{R}P^2$. The map $\m_2^{b(x),0,0}( - , \alpha)$ then counts rigid holomorphic strips between $\bL$ and $\mathbb{R}P^2$ with one extra boundary marked point mapping to the image of $\gamma$. From the description of the holomorphic strips in the proof of Theorem \ref{thm:mirror_rp2} it is clear that the only strips that intersect $\gamma$ are the two halves of the disk $[z:1:1]$, and in this case it intersects the curve exactly once. Hence $\m_2^{b(x),0,0}( - , \alpha)$ has the following matrix representation
		\begin{equation}\label{eq:2}
		\mathcal{F}^\bL(\alpha) = \left(\begin{array}{cccc}
			0    &0 &  0 & x^{-1}y^{-1}    \\
			0    &   0  & x^{-1} & 0    \\
			0    &   y^{-1}  & 0 & 0    \\
			1    &   0  & 0 & 0    \end{array}\right) 
	\end{equation}

Now consider \begin{equation*}
M = \left(\begin{array}{cccc}
	0    &0 &  0 & 0    \\
	x^{-1}    &   0  & 0 & 0    \\
	0    &   0  & 0 & x^{-1}y^{-1}    \\
	0   &   0  & 0 & 0    \end{array}\right) 
\end{equation*}
An elementary calculation shows that $[Q, M]= \mathcal{F}^\bL(\alpha) + x^{-1}{\mathrm{Id}}$. Hence, on cohomology we have $\mathcal{F}^\bL(\alpha)= x^{-1}{\mathrm{Id}}= x^{2}{\mathrm{Id}}$, as in the proof of Theorem \ref{thm:mf_rp2}. Now the proposition follows from Theorem \ref{thm:mf_rp2} and the description of $HF(\mathbb{R}P^2,b=0)$ given above. 
\end{proof}

There is a more conceptual explanation of the injectivity part of the above result. At this point it is still conjectural, but we give the outline here since this should be useful in more general settings, for example for $\mathbb{R}P^{2n}$.
In \cite{CLS}, the second author and his collaborators showed that the localized mirror functor is compatible with the \emph{open-closed maps}. This relied on results from \cite{FOOO_MS} which do not directly generalize to positive characteristic. This is the reason why the next couple of paragraphs are at this point conjectural. Nevertheless, we expect the end result to generalize to our setting, and provide us with the following commutative diagram:
\begin{equation}\label{diagram:oc}
	\begin{CD}
	 HF(\mathbb{R}P^{2n},b=0) @>\OC>>  QH^\bullet(\mathbb{C}P^{2n}) \\
	 @V  \mathcal{F}^\bL VV    @V I\circ\mathfrak{ks} VV      \\
	{\mathrm End}_{H(\mathcal{MF}^{un}(W))}(Q_{\mathbb{R}P^{2n}}) @>\OC_{\MF}>>  Jac(W)\frac{dx_1\wedge\ldots \wedge dx_{2n}}{x_{1}\cdots x_{2n}} 
\end{CD} \end{equation}
 Some explanations are in order: the horizontal maps are the open-closed maps in the Fukaya category (see \cite{She_fano}, where the notation $\OC^0$ is used) and in the category of matrix factorizations, also known as the boundary-bulk map. Although the map $\OC_{\MF}$ was defined in \cite{PV} for the category of $\mathbb{Z}/2$- graded factorizations, in characteristic zero, its construction still goes through in our setting. The vertical arrow on the right is the composition of the \emph{Kodaira-Spencer map} $\mathfrak{ks}: QH^\bullet(\mathbb{C}P^{2n}) \to Jac(W)$ (see \cite{FOOO_MS}) with a duality isomorphism $I$ which is an identity in this case (see \cite{CLS}). 

The map $\OC$ is dual to the \emph{closed-open} map $\mathcal{CO}: QH^\bullet(\mathbb{C}P^{2n})  \to HF(\mathbb{R}P^2,b=0)$, which was studied by Tonkonog in this setting \cite{ton18}. He showed that $\mathcal{CO}$ is an isomorphism. Therefore the map $\OC$ is also an isomorphism. It is well know that the Kodaira-Spencer map $\mathfrak{ks}$ is an isomorphism in this case, see for example \cite{FOOOT}. Therefore both compositions in the commutative diagram \ref{diagram:oc} are isomorphisms. This implies, in particular, that $\mathcal{F}^\bL$ is injective. 

\vspace{.1cm}

We now go back to our discussion of mirror symmetry for $\mathbb{R}P^2$. But first we need to recall some notions of $\AI$-category theory \cite{S08}. Given an $\AI$-category $\mathcal{C}$, it embeds, via the Yoneda embedding, into the category $mod(\mathcal{C})$ of right $\AI$ $\mathcal{C}$-modules. One then defines the derived category $D^\pi \mathcal{C}$ to be the smallest subcategory of $H(mod(\mathcal{C}))$ which contains the image of the Yoneda embedding and  is triangulated and split-closed. A full subcategory $\mathcal{B}\subset \mathcal{C}$ is said to split-generate $\mathcal{C}$ if the induced functor $D^\pi \mathcal{B} \to D^\pi \mathcal{C}$ is an equivalence. Please note that all these notions immediately extend to our ungraded, characteristic 2 setting. The main difference is that the shift functor is trivial.

Going back to our example: Tonkonog's result on the closed open for $\mathbb{R}P^2$, combined with Abouzaid's generation criterion \cite{A10} proves that $(\mathbb{R}P^2,b=0)$ (meaning the subcategory with just this object) split-generates the Fukaya category $\Fuk_0(\CP^2)$. This result, together with the general Homological Mirror Symmetry philosophy leads us to make the following conjecture.

\begin{conjecture}\label{conj}
	The object $Q_{\mathbb{R}P^2}$ split-generates the category $\MF^{un}(W=x+y+\frac{1}{xy})$.
\end{conjecture}

We can now state a version of Homological Mirror Symmetry for our example.

\begin{theorem}
	The localized mirror functor induces an embedding of the derived categories
	\[D^\pi\mathcal{F}^\bL: D^\pi\Fuk_0(\CP^2) \to D^\pi\MF^{un}(W=x+y+\frac{1}{xy}).\]
	If one assumes Conjecture \ref{conj}, this functor is an equivalence.
\end{theorem}
\begin{proof}
	Tonkonog's result \cite{ton18} that $(\mathbb{R}P^2,b=0)$ split-generates together with Proposition \ref{prop:FLiso} give the embedding in the statement. Conjecture \ref{conj} combined with Theorem \ref{thm:mirror_rp2} show that the image of $D^\pi\mathcal{F}^\bL$ split generates the target category which proves the second part of the statement.
\end{proof}

\section{Second example: Milnor fiber of $A_{2n-1}$ singularity}

Now, let us discuss the case of $A_{2n-1}$-singularity.
Let $M$ be the Milnor fiber of $A_{2n-1}$-singularity, given by $$\{(x,y) \in \C^2 \mid x^{2n} + y^2 =1\}$$
The diagonal symmetry group $G =\Z/2n \oplus \Z/2$ acts on $M$, and its quotient orbifold is denoted as $X=[M/G]$.
$X$ is a symplectic orbifold $\mathbb{P}^1_{2n, 2,\infty}$ which is a sphere with a puncture and two orbifold points (of $\Z/2n$ and $\Z/2$).

Seidel Lagrangian $\bL$, which is an immersed Lagrangian in $\mathbb{P}^1_{2n, 2,\infty}$ can be used to find the mirror.
Take the simplest arc between $\Z/2$ orbifold point and the puncture and denote it by $L$. 
Note that $L$ is in fact a Lagrangian sub-orbifold of $X$. In the uniformizing chart of $\Z/2$-point, $L$ may be taken as a $\Z/2$-invariant line
(See the red line in Figure \ref{A1} that is $\Z/2$-invariant).
But $\Z/2$-action reverses the orientation of the line, hence $L$ is not orientable. 
 The quotient of the Milnor fiber $\mathbb{P}^1_{2n, 2,\infty}$ has a mirror given by $W=x^{2n} + y^2 + xyz:\C^3 \to \C$.
In particular, wrapped Fukaya category of  $\mathbb{P}^1_{2n, 2,\infty}$ is quasi-equivalent to the matrix factorization category
of $W=x^{2n} + y^2 + xyz$ (see \cite{CCJ}).

Let us illustrate this in the case of $A_1$. Then the Milnor fiber $M$ is an annulus (centered at the origin of $\mathbb{R}^2$), and its $\Z/2 \oplus \Z/2$ action is defined as follows.
The diagonal action is nothing but a half-rotation of the annulus (without any fixed points).
Additional $\Z/2$-action can be taken as a simultaneous half-rotation at the two squares in Figure 1,
which gives two additional fixed points on the $x$-axis. The lifts of $\bL$ are given by the union of 4 circles.
$L$ is the intersection of the annulus with the positive $y$-axis. To avoid the triple intersection with $\bL$, we can perturb $L$ either
in a $\Z/2$-invariant way to obtain $L_2$ (which gives the wavy curve) or perturb it in a non-equivariant way to obtain $L_3$ (which gives the red line on the right).
Here we consider a  Floer theory of $L_i$ and $\bL$ in $\mathbb{P}^1_{2n, 2,\infty}$ or a $\Z/2 \oplus \Z/2$-equivariant Floer theory
of their lifts in the annulus.

\begin{figure}[h]
\includegraphics[scale=0.9]{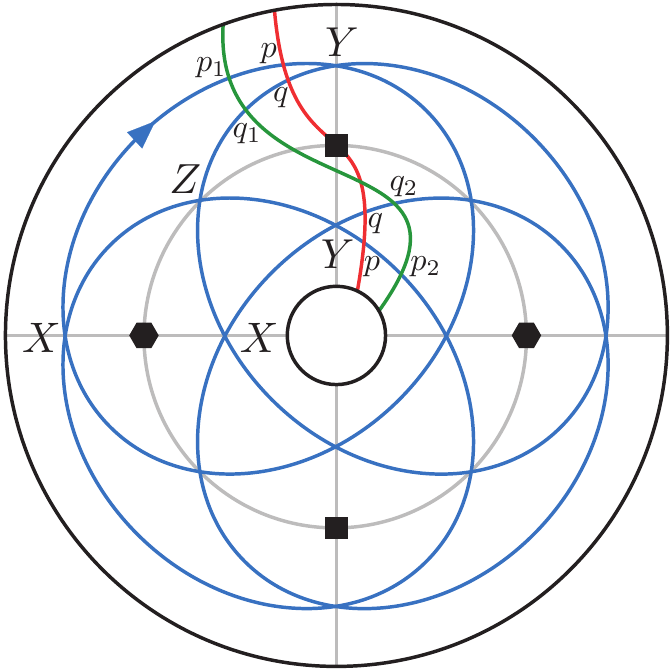}
\centering
\caption{$A_1$ Milnor fiber with $\Z/2$-invariant and $\Z/2$-equivariant Lagrangians}
\label{A1}
\end{figure}

We can apply the localized mirror functor of \cite{CHL17} to $\mathbb{C}P^1_{2n, 2,\infty}$.
\begin{prop}\label{prop:mf_an}
The localized mirror functor from the Fukaya category of the Milnor fiber quotient $X$ to matrix factorizations of $W=x^{2n} + y^2 + xyz$ was defined in \cite{CCJ}.
 Under the functor, non-orientable Lagrangian $L_2$ maps to the ungraded matrix factorization 
 $$Q= \left(\begin{array}{cc}
  x^n  &  y \\
 y +xz & x^n    
   \end{array}\right)$$
For the Berglund-H\"ubsch homological mirror symmetry, we only need to set $z=0$ (which is mirror to the monodromy of $A_{2n-1}$-singularity), to obtain the ungraded matrix factorization of $W_0=x^{2n} + y^2$:
    $$R= \left(\begin{array}{cc}
  x^n  &  y \\
  y &     x^n   
   \end{array}\right)$$
\end{prop}
\begin{proof}
Let us explain this in the case of $A_1$ and leave the general cases as an exercise.
Let us denote the $\Z/2$-equivariant perturbation of $L$ as $L_2$.
$L_2$ and the lift of the Seidel Lagrangian intersects at 4 points, and modulo $\Z/2$-action,
we label them as $p$ and $q$.

 We count decorated Floer strips between $L_2$ and $\bL$ (which
means, we allow corners of $X,Y,Z$ while counting strips from $p,q$ to $p,q$.
A bigon for the potential $y^2$ is cut by $L_2$ and it give rise to the entry $y$ from $p$ to $q$ and another entry $y$ 
from $q$ to $p$.  $XYZ$-triangle is cut by $L_2$ and it give rise to the additional entry $xz$ from $p$ to $q$.
A bigon for $x^2$ is cut by $L_2$ and it gives rise to the entry $x$ from $p$ to $p$.
(For general $n$,  an $x^{2n}$-gon is cut by $L_2$ to give the entry $x^n$).
This gives the matrix $Q$.
\end{proof}
\begin{remark}
Here $L_3$ may be regarded as an orientation double cover of $L_2$ in the sense that
we take a $\Z/2$-equivariant (not invariant) perturbation of $L$.
In  Figure \ref{A1}, $L_3$ is drawn as a   slight perturbation of $L_2$ and 
 $L_3 \cap \bL$ are labeled by $p_1,q_1, q_2,p_2$ as in the figure.
 Now $L_3 \cap \bL$ can be $\Z/2$-graded and $p_1$ and $q_2$ have the same $\Z/2$-grading, which is
the opposite of the grading of $p_2$ and $q_1$. One can check that the resulting matrix factorization is $\mathfrak{D}(Q)$.
\end{remark}

We conclude this section by describing the endomorphism algebras of the matrix factorizations in the previous proposition. Its proof follows from two straightforward computations that we omit. 

\begin{prop}
 Let $Q$ be the object in $\MF^{un}(W=x^{2n} + y^2 + xyz)$ defined in Proposition \ref{prop:mf_an}, then
 \[{\mathrm End}_{H(\mathcal{MF}^{un}(W))}(Q)\cong \frac{\K[x,z]}{xz}.\]
 
 For the matrix factorization $R$ in $\MF^{un}(W_0=x^{2n} + y^2)$, also defined in  Proposition \ref{prop:mf_an}, we have the following isomorphism
 \[{\mathrm End}_{H(\mathcal{MF}^{un}(W_0))}(R)\cong \frac{\K[x,J]}{J^2-1}.\]
\end{prop}

In particular, this proposition shows that the \emph{closed-open map} $\mathcal{C}$ from Theorem \ref{thm:mf_rp2} is not an isomorphism in these cases.

\section{Proof of Theorem \ref{thm:mf_rp2}}

On this section, for simplicity we will write $Q$ for the matrix factorization $Q_{\mathbb{R}P^2}$. We start by defining the matrices $U$ and $V$ as follows:
$$Q = \left(\begin{array}{cccc}
 0    &1 &  1 & x^{-1}y^{-1}    \\
 y    &   0  & x^{-1} & 1    \\
 x    &   y^{-1}  & 0 & 1    \\
  1    &   x  & y & 0    
   \end{array}\right) = \left(\begin{array}{cc} U & V \\ xV & U \end{array}\right)$$

 where $U =  \left(\begin{array}{cc}  0 & 1\\y & 0 \end{array}\right)$ and
 $V =  \left(\begin{array}{cc} 1  & x^{-1}y^{-1} \\x^{-1} & 1 \end{array}\right) = \Id + x^{-1}y^{-1} U$.
 
 Given $f\in {\mathrm End}_{\mathcal{MF}^{un}(W)}(Q)$ we denote  its differential by $\delta_Q(f) = [Q,f] = Qf + fQ$.
 Similarly, since $U^2 = y\mathrm{Id}$, we will consider the differential $\delta_U$ (on the space of 2 by 2 matrices).
 
 Our first goal is to characterize morphisms commuting with $Q$, or $\delta_Q (f) =0$. We will need a few lemmas.
 
 \begin{lemma}
 	We have the following:
 \begin{enumerate}
 \item $\delta_U$ is acyclic.
 \item $U^2 = y \Id, V^2 = ( 1+ x^{-2}y^{-1}) \Id, \;\; VU = UV = x^{-1} \Id + U$.
 \item  For $A =  \left(\begin{array}{cc}  a & b \\c & d \end{array}\right)$,
$$\delta_U (A) =  \left(\begin{array}{cc}   at(A) & tr(A) \\ y \cdot tr(A) & at(A) \end{array}\right),$$
where $tr(A)  = a+d, at(A) = yb +c$.
\end{enumerate}
\end{lemma}
\begin{proof}
Second and third statement can be checked directly by simple computations.
We can use the third statement to prove the first one. 
For a matrix $A$ with $\delta_U(A)=0$, we have $tr(A) = at(A)=0$, and
hence $A$ is of the form $A =  \left(\begin{array}{cc}  a & y^{-1}c \\ c & a \end{array}\right).$
But such an $A$ can be written as $\delta_U   \left(\begin{array}{cc}  0 & 0 \\ x & y^{-1}c \end{array}\right)$.
\end{proof}

We write $f =  \left(\begin{array}{cc}  A & B \\C & D \end{array}\right)$,
and compute
\begin{equation}\label{eq:deltaQf}
 \delta_Q (f) = \left(\begin{array}{cc}  [U,A] +VC +xBV & [U,B]+VD+AV\\ \;[U,C] + x(VA+DV) & [U,D] + xVB +CV \end{array}\right)
 \end{equation}
 The vanishing of $(1,2)$ component of the above implies that 
  $$VD+AV=[U,B] =  \delta_U(B)$$
  Hence, $\delta_U(VD+AV)=0$, which implies that 
\begin{equation}\label{eq:trati}
tr(VD+AV) = at(VD+AV)=0.
\end{equation}
The following lemma can be checked easily.
\begin{lemma}
For any $(2\times 2)$ matrix $F$, we have
$$tr(VF) = tr(FV) = tr(F) + x^{-1}y^{-1} \cdot at(F)$$
$$at(VF) =tr(FV) = at(F) + x^{-1} \cdot tr(F)$$
\end{lemma}
Using this lemma, we can prove the following
\begin{lemma}
For any matrix $(2\times 2)$ matrix $F$ satisfying $tr(VF)=at(VF) =0$, we have
$tr(F) = at(F)=0$ and hence
$$ [U, F] =0$$
\end{lemma}
\begin{proof}
From the previous lemma, we have 
$$tr(F)= x^{-1}y^{-1} \cdot at(F),  at(F) = x^{-1} \cdot tr(F)$$
Hence, $tr(F) = x^{-2}y^{-1} \cdot tr(F)$, which implies that $tr(F) = at(F)=0$.
\end{proof}

As a corollary, \eqref{eq:trati} implies that
$$tr(VD+AV) = tr(VD)+ tr(AV) = tr(VD) +tr(VA)  = tr(V(A+D))$$
and similarly $at(VD+AV) = at(V(A+D))$.
Thus, we have $[U, A+D]=0$ or 
\begin{equation}\label{eq:tratii}
D = A + [U,T]
\end{equation}
 for some $T$ since $\delta_U$ is acyclic.

Similarly, from the (1,1) component of the expression \eqref{eq:deltaQf}, we have
$VC+xBV$ is in the image of $\delta_U$, and hence $tr(VC+xBV)= at(VC+xBV)=0$.
Applying the same argument, we get
$[U,C] = [U,xB]$  or  
\begin{equation}\label{eq:tratii}
C= xB + [U,S]
\end{equation}
for some $S$.

Therefore, $\delta_Q(f)=0$ implies that $D=A+[U,T], C=xB+[U,S]$.
If we set $$f =  \left(\begin{array}{cc}  A & B \\xB+ [U,S] &A+[U,T] \end{array}\right),$$
the equation $\delta_Q(f)$ becomes
$$[Q,f] = \left(\begin{array}{cc}  [U,A]+x[V,B] +V[U,S] & [U,B] +[V,A] +V[U,T] \\
x[U,B] +x[V,A] +xV[U,T] &   [U,A]+x[V,B] +V[U,S] \end{array}\right)$$
$$ =  \left(\begin{array}{cc}  [U,A+y^{-1}B] +V[U,S] & [U,B+x^{-1}y^{-1}A] +V[U,T] \\
x([U,B +x^{-1}y^{-1}A] +V[U,T]) &   [U,A+y^{-1}B] +V[U,S]  \end{array}\right).$$
Note that $[V,[U,S]] = [U,[U,S]]=0$ from the identity $V= \Id + x^{-1}y^{-1}U$,
which implies that $V[U,S] = [U,S]V$.

Thus $\delta_Q(f)=0$ is equivalent to 
\begin{eqnarray}
[U,A+y^{-1}B]  &=& V[U,S] \\
\; [U,B +x^{-1}y^{-1}A] &=&V[U,T] 
\end{eqnarray}
or equivalently,  $A,B,S,T$ satisfy the following {\em closed-conditions}:
\begin{eqnarray}\label{eq:closecond}
at(A+y^{-1}B)  &=& at(S) + x^{-1} tr(S)  \\
tr(A+y^{-1}B)  &=& tr(S) + x^{-1}y^{-1} at(S) \nonumber \\ 
at(B +x^{-1}y^{-1}A)  &=& at(T) + x^{-1} tr(T)  \nonumber \\
tr(B +x^{-1}y^{-1}A)  &=& tr(T) + x^{-1}y^{-1} at(T) \nonumber
\end{eqnarray}
After elementary operations, this is equivalent to
\begin{eqnarray}
at(A+y^{-1}B)  &=& at(S) + x^{-1} tr(S)  \\
at(A+y^{-1}B) + x^{-1}tr(A+y^{-1}B)  &=& (1 + x^{-2}y^{-1} )at(S)  \nonumber\\
at(B +x^{-1}y^{-1}A) +x^{-1}tr(B +x^{-1}y^{-1}A)  &=& (1+ x^{-2}y^{-1}) at(T) \nonumber \\
tr(B +x^{-1}y^{-1}A)  &=& tr(T) + x^{-1}y^{-1} at(T) \nonumber
\end{eqnarray}
Here, by combining the second and third equation, we may replace the second condition by
\begin{equation}\label{eq:st}
 \big(at(B)+x^{-1}tr(B)\big)(1+x^{-1}y^{-2}) = 
(1 + x^{-2}y^{-1} )\big(at(T) + x^{-1}y^{-1} at(S) \big).
\end{equation}

This implies that $(1+x^{-2}y^{-1})$ divides 
$at(B)+x^{-1}tr(B) = x^{-1}(b_1+b_4) + yb_2 + b_3$
where $B = \left(\begin{array}{cc}  b_1 & b_2 \\b_3&b_4 \end{array}\right)$.
Hence, $b_3 = x^{-1}(b_1+b_4) + yb_2 + p (1+x^{-2}y^{-1})$
for some $p \in \mathbb{K}[x^{\pm1},y^{\pm 1}]$.

Therefore, we conclude:
\begin{prop}\label{prop:deltaQ}	\hfill
\begin{enumerate}
\item Let $f = \left(\begin{array}{cc}  A& B \\C&D \end{array}\right)$
be a closed element, that is $\delta_Q(f)=0$. Then
$D=A +[U,T]$, $C=xB+[U,S]$ for some $S,T$ and
these satisfy the closed conditions \eqref{eq:closecond}.

\item Any element $f$ with $\delta_Q(f)=0$ is cohomologous
to a multiple of an identity matrix, $\alpha \cdot \Id$.
\end{enumerate}
\end{prop}
\begin{proof}
 The first statement follows from the previous discussion. For the second statement we first define

 $A'  = \left(\begin{array}{cc}  0 & 0 \\ yb_2& b_4+ x^{-1}y^{-1}p \end{array}\right)$.
 $B' = \left(\begin{array}{cc}  0 &0 \\ x^{-1}y^{-1}p & 0 \end{array}\right)$.
 $C'=\left(\begin{array}{cc}  y^{-1}b_1+y^{-1}b_4+x^{-1}y^{-2}p & 0 \\ 0& 0 \end{array}\right)$.
 $D'=\left(\begin{array}{cc}  b_1 & b_2 \\ p& 0 \end{array}\right)$.
 Where the $b_i$ and $p$ are as above.
 Then it is easy to check that
$$[Q, \left(\begin{array}{cc}  A' & B' \\C'&D' \end{array}\right)] =
\left(\begin{array}{cc}  * &B \\xB&* \end{array}\right),
[Q, \left(\begin{array}{cc}  0 & 0 \\S&0\end{array}\right)] =
\left(\begin{array}{cc}  0 &0 \\ \;[U,S]&0 \end{array}\right).$$
Therefore,  any $f$ with $\delta_Q(f)=0$
is cohomologous to an element of the form 
\begin{equation}\label{eq:diag1}
\left(\begin{array}{cc}  A &0 \\0& D \end{array}\right).
\end{equation}

An element of the form \eqref{eq:diag1} is $\delta_Q$-closed
if and only if $D =A+ [U,T]$ and the closed-conditions
$$at(A) = tr(A) =0, at(T) = x^{-1}tr(T), tr(T) = x^{-1}y^{-1}at(T),$$
are satisfied.
This implies also that $at(T)=tr(T)=0$, or $[U,T]=0$.
Therefore, we have $A=D$ with $at(A) =tr(A)=0$,
and hence $A$ is of the form 
$\left(\begin{array}{cc}  a_1 & a_2 \\ ya_2& a_1 \end{array}\right)$.

Next, note that if we consider $C=\left(\begin{array}{cc}  0& a_2 \\ya_2&0 \end{array}\right)$, then 
$\delta_Q( \left(\begin{array}{cc}  0 &0 \\ C& 0 \end{array}\right))
=\left(\begin{array}{cc}  F & 0 \\0&F \end{array}\right)$
for $F$ of the type $\left(\begin{array}{cc}  *& a_2 \\ya_2&* \end{array}\right)$.
Hence we conclude that $f$ is cohomologous to 
a multiple of an identity matrix $\alpha \cdot \Id$.
\end{proof}

The next proposition completes the proof of Theorem \ref{thm:mf_rp2}.

\begin{prop}
Let $\mathcal{C}: Jac(W) \to {\mathrm End}_{H(\mathcal{MF}^{un}(W))}(Q_{\mathbb{R}P^2})$ be the map sending $\alpha \to \alpha \cdot \Id$. This map is an isomorphism.
\end{prop}
\begin{proof}
First we note that Equation \ref{eq:Qdiff} implies that the map $\mathcal{C}$ is well-defined.
Surjectivity follows from Proposition \ref{prop:deltaQ}.

To show that $\mathcal{C}$ is an isomorphism, we prove that
if $\alpha \cdot \Id = \delta_Q(f)$ for some $f$, then $\alpha$
is an element of Jacobian ideal.
From the equation \eqref{eq:deltaQf}, and proceeding as in the previous case,
we easily see that
$VD+AV$ is in the image of $\delta_U$ or $D = A+ [U,T]$ for some $T$, and
$VC+xBV$ is in the image of $\delta_U + \alpha \cdot \Id$.
But since $\alpha \cdot \Id $ is also $\delta_U$ exact (as it is $\delta_U$-closed),
$VC+xBV$ is in the image of $\delta_U$ and hence we may set
$C= xB+[U,S]$ as before.

We also have the closed conditions \eqref{eq:closecond},
where the first equation should be modified as
$at(A+y^{-1}B) = at(S) + x^{-1} tr(S) + \alpha$.

The equation \eqref{eq:st} is modified as
\begin{equation}\label{eq:st}
 (at(B)+x^{-1}tr(B))(1+x^{-1}y^{-2}) = 
(1 + x^{-2}y^{-1} )(at(T) + x^{-1}y^{-1} at(S)) + x^{-1}y^{-1}\alpha
\end{equation}
or equivalently,
\begin{equation}\label{eq:st}
\alpha = xy (at(B)+x^{-1}tr(B))(1+x^{-1}y^{-2}) +  
xy(1 + x^{-2}y^{-1} )(at(T) + x^{-1}y^{-1} at(S)).
\end{equation}
Hence, $\alpha$ lies in the ideal generated by 
$\partial_xW = (1+x^{-1}y^{-2})$ and $ \partial_y W =(1+x^{-2}y^{-1})$.
\end{proof}

\bibliographystyle{amsalpha}
\bibliography{geometry}

\end{document}